\documentclass{amsart}

\usepackage{amssymb}
\usepackage{amsmath}
\usepackage{tikz}
\usepackage{hyperref}

\newtheorem{theorem}{Theorem}[section]
\newtheorem{lemma}[theorem]{Lemma}
\newtheorem{proposition}[theorem]{Proposition}

\theoremstyle{remark}
\newtheorem{remark}[theorem]{Remark}

\newcommand{\sh}{\mathcal}
\newcommand{\modu}{\mathcal}

\newcommand{\Pic}{\operatorname{Pic}}

\begin{document}
	
	\title[Interpolation on K3 and abelian surfaces]{Interpolation of fat points on K3 and abelian surfaces}
	
	\author[A. Zahariuc]{Adrian Zahariuc}
	\date{}
	\address{Department of Mathematics and Statistics, University of Windsor, 401 Sunset Ave, Windsor, ON, N9B 3P4, Canada}
	\email{adrian.zahariuc@uwindsor.ca}
	\thanks{\includegraphics[scale=0.3]{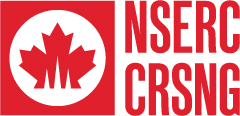} We acknowledge the support of the Natural Sciences and Engineering Research Council of Canada (NSERC), RGPIN-2020-05497. Cette recherche a \'{e}t\'{e} financ\'{e}e par le Conseil de recherches en sciences naturelles et en g\'{e}nie du Canada (CRSNG), RGPIN-2020-05497.}
	
	\maketitle
	
	\begin{abstract}
	We prove that any number of general fat points of any multiplicities impose the expected number of conditions on a linear system on a smooth projective surface, in several cases including primitive linear systems on very general K3 and abelian surfaces, `Du Val' linear systems on blowups of ${\mathbb P}^2$ at $9$ very general points, and certain linear systems on some ruled surfaces over elliptic curves. This is done by answering a question of the author about the case of only one fat point on a certain ruled surface, which follows from a circle of results due to Treibich--Verdier, Segal--Wilson, and others.
	\end{abstract}
	
	\section{Introduction}\label{section: introduction}
	
	The Segre--Harbourne--Gimigliano--Hirschowitz (SHGH) Conjecture \cite{[Se60], [Ha86], [Gi87], [Hi89]} is a well-known open problem in algebraic geometry. Please see \cite{[BM04]} for a survey of the history and various approaches to the conjecture. Questions similar to the SHGH Conjecture on projective surfaces other than ${\mathbb P}^2$ have been studied, for instance, in \cite{[DL05], [DL05a], [Du10], [Hu13], [La02], [La11], [Za19]}. Such questions also relate to Seshadri constants, though in general they require even more specific (non-asymptotic) information.
		
	This note is a follow-up to \cite{[Za19]}. We will show that, over the complex numbers, \cite[Conjecture 2.5]{[Za19]} can be deduced from certain facts (of a rather analytic nature), thereby establishing the results proved conditionally on this claim in \cite{[Za19]}. First, we state these results clearly and extend slightly the range of situations covered.
	
	\begin{theorem}\label{theorem: main theorem}
		Let $S$ be a smooth complex projective surface and ${\sh L}$ an invertible ${\sh O}_S$-module, such that $(S,{\sh L})$ is one of the following:
		\begin{enumerate}
		\item\label{situation item: K3} a very general primitively polarized K3 surface of any degree;
		\item\label{situation item: abelian surface} a very general primitively polarized abelian surface of any degree (i.e. a very general $(1,d)$-polarized abelian surface);
		\item\label{situation item: Du Val} the blowup of ${\mathbb P}^2$ at $9$ very general points, and a Du Val linear system, that is,
		$ c_1({\sh L}) = -kK_S + E$, where $E \simeq {\mathbb P}^1 \subset S$ is a $(-1)$-curve and $k$ is a positive integer;
		\item $S = {\mathbb P}{\sh E}$, where ${\sh E}$ is a rank $2$ locally free sheaf on a smooth projective curve $E$ of genus $1$, such that either
		\begin{enumerate}
		\item \label{situation item: split ruled} ${\sh E}={\sh O}_E \oplus {\sh J}$ with ${\sh J} \in \Pic^0(E)$ very general; or
		\item \label{situation item: Atiyah ruled} ${\sh E}$ is the unique nontrivial extension of ${\sh O}_E$ by ${\sh O}_E$,
		\end{enumerate}
		and ${\sh L}$ is any ample line bundle with intersection number $1$ with the section of $S$ corresponding to the natural surjection ${\sh E} \twoheadrightarrow {\sh O}_E$ (in either case).
		\end{enumerate}
	Let $p_1,\ldots,p_n \in S$ be $n$ general points, $m_1,\ldots,m_n$ positive integers, and 
	\begin{equation}\label{equation: notation for linear system}
	\left|{\sh L} ({m_1} ,\ldots, {m_n}) \right| = {\mathbb P}H^0(S,{\sh L} \otimes {\sh I}_{p_1,S}^{m_1} \otimes \cdots \otimes {\sh I}_{p_n,S}^{m_n}) \subset |{\sh L}| 
	\end{equation}
	the linear system of divisors in $|{\sh L}|$ which have multiplicity at least $m_i$ at $p_i$, for all $i=1,2\ldots,n$. Then
	$$ \dim \left|{\sh L} ({m_1} ,\ldots, {m_n}) \right| = \max\left\{ -1, \dim |{\sh L}| - \sum_{i=1}^n \frac{m_i(m_i+1)}{2} \right\}, $$
	with the convention that empty linear systems have dimension $-1$.
	\end{theorem}
	
	To clarify, in cases \ref{situation item: K3} and \ref{situation item: abelian surface}, $|{\sh L}|$ is the \emph{primitive} linear system. For K3 surfaces, we confirm the conjectures of De Volder and Laface (Conjectures 2.1 and 2.3 in \cite{[DL05a]}) in the primitive case ($d=1$ with the notation in \cite{[DL05a]}), please see Remark \ref{remark: De Volder Laface}. Case \ref{situation item: Du Val} can equally well be stated in terms of ${\mathbb P}^2$, so we also obtain some cases of the SHGH Conjecture. Cases \ref{situation item: abelian surface} and \ref{situation item: split ruled} were not considered explicitly in \cite{[Za19]}, but they can be dealt with using the same idea, and seem natural enough to include. We have also removed the generality assumption on $E$ in \ref{situation item: Atiyah ruled}. To the author's knowledge, Theorem \ref{theorem: main theorem} is currently the only SHGH-type result with no restrictions on the number and weights of fat points.
		
	Using the arguments in \cite{[Za19]}, Theorem \ref{theorem: main theorem} boils down to case \ref{situation item: Atiyah ruled} of Theorem \ref{theorem: main theorem} and $n=1$, or specifically, the characteristic $0$ part of \cite[Conjecture 2.5]{[Za19]}. The reduction is discussed in \S\ref{section: proof of the main result}. 
	
	The conjecture in \cite{[Za19]} is stated as Proposition \ref{proposition: former conjecture} (in the form which removes the generality assumption on the $j$-invariant), and \S\ref{section: proof of the conjecture} is devoted to proving it. It turns out that the claim can be deduced from facts about \emph{tangential covers} of elliptic curves, that were studied by Treibich, Verdier, and other authors \cite{[Tr89], [Tr93a], [Tr93b], [TV89], [TV91], [Se23], [FT17]}. From the point of view of SHGH-type problems, the situation is quite unexpected, since the key fact comes from analytic work of Segal and Wilson \cite{[SW85]}, and it seems that no algebraic proofs are currently available. 
		
	\subsection*{Acknowledgements} I would like to thank Xi Chen, Brian Osserman, Edoardo Sernesi, and Armando Treibich for useful discussions. I am especially grateful to Sernesi and Treibich for invaluable help with navigating the literature.

	\section{Proof of the conjecture from \cite{[Za19]}}\label{section: proof of the conjecture}
	
	\subsection{Finiteness of $\theta$ vanishing in the KP direction}
		
		Let $C$ be an integral (possibly singular) complex projective curve of arithmetic genus $g \geq 1$, and $p \in C$ a smooth point. There is a canonical identification of the tangent line $T_pC$ to $C$ at $p$ with ${\sh O}_C(p) \otimes {\sh O}_p$. The exact sequence
		$$ 0 \longrightarrow {\sh O}_C \longrightarrow {\sh O}_C(p) \longrightarrow {\sh O}_C(p) \otimes {\sh O}_p \longrightarrow 0 $$
		gives a natural embedding $T_pC \hookrightarrow H^1({\sh O}_C)$ as the first coboundary map in the associated long exact sequence, whose image thus coincides with the kernel of the map $H^1({\sh O}_C) \to H^1({\sh O}_C(p))$. Since $ {\sh T}_{\Pic(C)} \cong H^1(C,{\sh O}_C) \otimes {\sh O}_{\Pic(C)}$, 
		the point $p$ induces a natural direction at any point in $\Pic(C)$, which is sometimes called the \emph{Kadomtsev-Petviashvili (KP) direction}.
		
		The essential ingredient in our arguments is the following.
		
	\begin{theorem}\label{theorem: SW}
	Let $\Delta \hookrightarrow \mathrm{Pic}^{g-1}(C)$ be a holomorphic map from an open disk, such that the tangent line to any point of $\Delta$ coincides with the KP direction $T_pC \subset H^1({\sh O}_C)$. Then there exists ${\sh L} \in \Delta$ such that $H^0(C,{\sh L}) = 0$.
	\end{theorem}
	
	\begin{proof}
	Follows from \cite[Proposition 8.6]{[SW85]} and the `Krichever dictionary', as stated and explained in \cite[\S3.4]{[Tr93a]}. Indeed, after compactifying the Jacobian in order to have the same setup, the first two lines on page 45 of \cite{[Tr93a]} imply our claim, since we are claiming precisely that $\Delta$ is not contained in the theta divisor. Note also that Proposition 3.7 on the same page provides the Segal--Wilson formula for the vanishing/contact order.
	\end{proof}
	
	\begin{remark} We will only require the case when $C$ has at worst planar singularities, in which the theory is somewhat simpler \cite[p. 38]{[SW85]}. In fact, one can even make a certain plausible conjecture later, which would ensure that we only invoke Theorem \ref{theorem: SW} for $C$ smooth. When $C$ is smooth, Theorem \ref{theorem: SW} also follows from results of Fay \cite{[Fa84]}. However, it is stated in \cite[footnote to p. 326]{[BV03]} that a direct geometric proof is not known even this case.
	\end{remark}
	
	\subsection{Conjecture 2.5 in \cite{[Za19]} in characteristic $0$}\label{S: 2.2}
	
	Let $E$ be a smooth complex projective curve of genus $1$, $\rho:S = {\mathbb P} {\sh V} \to E$ the ruled surface over $E$, where ${\sh V}$ is the unique rank $2$ vector bundle that fits in a non-split short exact sequence 
	$$ 0 \longrightarrow {\sh O}_E \longrightarrow {\sh V} \longrightarrow {\sh O}_E \longrightarrow 0, $$
	$E_\infty \subset S$ the section corresponding to ${\sh V} \to {\sh O}_E \to 0$, $q \in E$, and $F_s=\rho^{-1}(s)$, for any $s \in E$. There is a simple way to characterize curves which occur in $|F_q+kE_\infty|$ due to Treibich and Verdier \cite{[TV91]}: the projection map from the curve to $E$ must be so-called \emph{tangential}. We state the result in a form closer to \cite[Lemma 3.3]{[Za22]} (the author was unfortunately unaware of the literature when \cite{[Za22]} was written).
	
	\begin{lemma}[essentially {\cite[Corollaire 3.10]{[TV91]}}]\label{lemma: tangential}
		Let $C$ a projective integral curve, and $f:C \to S$ a morphism such that $f^{-1}(E_\infty) = \{q\}$ scheme-theoretically, and $q$ is a nonsingular point of $C$. Then the composition
		$$ H^1(E,{\sh O}_E) \xrightarrow{(\rho f)^*} H^1(C,{\sh O}_C) \longrightarrow H^1(C,{\sh O}_C(q)) $$
		is equal to $0$.
	\end{lemma}
	
	\begin{proof}
		As in the proof of \cite[Lemma 3.3]{[Za22]}, the map $H^1(S,{\sh O}_S) \to H^1(S,{\sh O}_S(E_\infty))$ induced by the inclusion ${\sh O}_S \subset {\sh O}_S(E_\infty)$ is equal to $0$. The commutativity of
			\begin{center}
			\begin{tikzpicture}[scale = 1]
				\node (nww) at (-3,1.3) {$H^1(E,{\sh O}_E)$};
				\node (sw) at (0,0) {$H^1(S,{\sh O}_S(E_\infty))$};
				\node (nw) at (0,1.3) {$H^1(S,{\sh O}_S)$};
				\node (se) at (3,0) {$H^1(C,{\sh O}_C(q))$};
				\node (ne) at (3,1.3) {$H^1(C,{\sh O}_C)$};
				\draw [->] (nww) to node[midway, above] {$\rho^*$} (nw); 
				\draw [->] (nw) to node[midway, above] {$f^*$} (ne); 
				\draw [->] (ne) to (se);
				\draw [->] (nw) to node[midway, right] {$0$} (sw); 
				\draw [->] (sw) to (se);
			\end{tikzpicture}
		\end{center}
		completes the proof.
	\end{proof}
	
	\begin{proposition}[{\cite[Conjecture 2.5]{[Za19]}}]\label{proposition: former conjecture}
		Let $x \in S$ be a general (closed) point and $m$ a positive integer. Let $k$ be the minimum positive integer for which the linear system $|F_q+kE_\infty|$ contains a curve of multiplicity at least $m$ at $x$. Then
		$$ k = \frac{m(m+1)}{2}. $$
	\end{proposition}
	
	Here is a casual summary of the proof. Instead of having $q$ fixed and $x$ variable, we may equally well consider $x$ fixed and $q$ variable, thanks to the automorphisms of $S$. Imagine a hypothetical family of high multiplicity curves at $x$ intersecting $E_\infty$ at the variable point $q$, which we partially normalize by blowing up at $x$. The family of effective divisors cut by this family of curves on one of the curves in the family has certain features that contradict properties of tangential covers obtained by combining Lemma \ref{lemma: tangential} and Theorem \ref{theorem: SW}.
	
	\begin{proof}
	It is obvious that $k \leq {m+1 \choose 2}$, since $h^0({\sh O}_S(F_q+kE_\infty)) = k+1$ by e.g. \cite[Proposition 2.3]{[Za19]}. Indeed, the map
	$$ H^0(S,{\sh O}_S(F_q+kE_\infty)) \longrightarrow H^0(S,{\sh O}_S(F_q+kE_\infty) \otimes {\sh O}_S/{\sh I}_{x,S}^m) $$
	connot be injective as soon as $k \geq {m+1 \choose 2}$ for obvious dimension reasons. 
	
	Assume by way of contradiction that the inequality was strict, and let 
	$$ r = {m+1 \choose 2} - k -1 \geq 0. $$
	Let $Y \in |F_q+kE_\infty|$ of multiplicity at least $m$ at $x$. Since all divisors in $|F_q+kE_\infty|$ are sums of an integral curve with a multiple of $E_\infty$ (in characteristic $0$, e.g. \cite[\S2.2]{[Za19]}), $Y$ must be integral by the minimality assumption on $k$. Moreover, by the genus-degree formula,
	\begin{equation}\label{equation: simple estimate}
		k= p_a(Y) \geq p_g(Y) + {\mathrm{mult}_xY \choose 2} \geq 1+ {\mathrm{mult}_xY \choose 2},
	\end{equation}
	so the multiplicity of $Y$ at $x$ must be precisely $m$. 
	
	Let $\beta:S' \to S$ be the blowup of $S$ at $x$ with exceptional curve $W \subset S'$, $C \subset S'$ the proper transform of $Y$, and $f:C \to S$ the restriction of $\beta$ to $C$. Let $g=p_a(C)$. We have $C \sim \beta^*Y - mW$ on $S'$, so
	\begin{equation}\label{equation: C^2} C^2 = (\beta^*F_q+k\beta^*E_\infty - mW)^2 = -m^2 +2k \end{equation}
	and
	$$ C \cdot K_{S'} =  (\beta^*F_q+k\beta^*E_\infty - mW) \cdot (-2\beta^*E_\infty + W) = m-2. $$
	Hence, by the genus-degree formula,
	\begin{equation}\label{equation: genus of C} 
	g = 1 + \frac{C^2+C\cdot K_{S'}}{2} = k-{m \choose 2} = m - 1 - r. 
	\end{equation}
	It is also clear that $g \geq 1$ since $C$ maps non-constantly to $E$.
	
	We claim that for any $s \in E$, the linear system $|F_{s}+kE_\infty|$ contains a divisor $Y(s,x)$ of multiplicity $m$ at $x$. Let $y \in S \backslash E_\infty$ such that 
	\begin{equation}\label{equation: choice of other point}
	\rho(y) = \rho(x) + s-q \in A_0(E).
	\end{equation}
	Recall that the subgroup of $\mathrm{Aut}(S)$ consisting of automorphisms which lie above translation automorphisms of $E$ acts transitively on $S \backslash E_\infty$, e.g. \cite[\S2.1]{[Za19]}, and choose such an automorphism $\psi$ such that $\psi(y) = x$. For $y$ general, and therefore for $y$ arbitrary by semi-continuity, there exists a divisor $Y(q,y) \in |F_q+kE_\infty|$ with multiplicity at least $m$ at $y$ by assumption. Then we define 
	$$ Y(s,x) = \psi(Y(q,y)) \subset S, $$ 
	and note that \eqref{equation: choice of other point} implies that $Y(s,x) \in |F_{s}+kE_\infty|$. Moreover, the multiplicity of $Y(q,y)$ at $y$, and therefore of $Y(s,x)$ at $x$, must be precisely $m$ by arguments similar to the ones above, cf. \eqref{equation: simple estimate}. (For special $s$, it could happen in principle that $Y(q,y)$ contains $E_\infty$, but it doesn't really matter; we may consider the analogue of \eqref{equation: simple estimate} for the irreducible component different from $E_\infty$.) Then, if $Y'(s,x) \subset S'$ is the proper transform of $Y(s,x)$, we have
	\begin{equation}\label{equation: Y' class} Y'(s,x) \sim \beta^*Y(s,x) -mW \sim  \beta^*F_s+k\beta^*E_\infty - mW. \end{equation}
	Let $p_1,\ldots,p_r \in C$ be $r$ arbitrary nonsingular points on $C$, and let
	\begin{equation}\label{equation: definition of L_s}
	{\sh L}_s = \phi^*{\sh O}_E(s) \otimes {\sh O}_C(kq+p_1+\cdots+p_r) \otimes ({\sh O}_{S'}(-mW)|_C),
	\end{equation}
	where $\phi = \rho \circ f$. This definition makes clear that $\{{\sh L}_s:s \in E\}$ is a curve in $\mathrm{Pic}(C)$, though we prefer to think of it as 
	\begin{equation}\label{equation: effective L_s} {\sh L}_s = {\sh O}_C(p_1+\cdots+p_r) \otimes f^* {\sh O}_{S'}(Y'(s,x)), \end{equation}
	which is equivalent by \eqref{equation: Y' class}. Then $h^0({\sh L}_s) > 0$ by \eqref{equation: effective L_s} (for $s \neq q$, but in particular also for $s=q$ by semi-continuity), and
	\begin{equation*} \deg {\sh L}_s= C \cdot Y'(s,x) + r = C^2+r = 2k -m^2 + r= m-2-r = g-1 \end{equation*}
	by \eqref{equation: effective L_s}, \eqref{equation: C^2} and \eqref{equation: genus of C}. 
	
	To summarize, the curve $ Z: = \{ {\sh L}_s: s \in E \} \subset \mathrm{Pic}^{g-1}(C) $ is a family of \emph{effective} degree $g-1$ Cartier divisors on $C$, and at the same time a coset of $\phi^*(\mathrm{Pic}^0(E))$ in $\mathrm{Pic}(C)$ by \eqref{equation: definition of L_s}. It follows that
	$$ T_{[{\sh L}_s]}Z \subset T_{[{\sh L}_s]} \mathrm{Pic}^{g-1}(C) = H^1(C,{\sh O}_C) $$ 
	is the image of $H^1(E,{\sh O}_E) \to H^1(C,{\sh O}_C)$, which is in fact the KP direction
	$$ T_qC = H^0(C,{\sh O}_C(q) \otimes {\sh O}_q) \hookrightarrow H^1(C,{\sh O}_C) $$ 
	by Lemma \ref{lemma: tangential}. This contradicts Theorem \ref{theorem: SW}, since $h^0({\sh L}_s)>0$ for all $s \in E$.
	\end{proof}
	
	Note that Proposition \ref{proposition: former conjecture} not only proves \cite[Conjecture 2.5]{[Za19]} in characteristic $0$, but strengthens it by removing the generality assumption on the $j$-invariant.
	
	\section{Proof of the main result}\label{section: proof of the main result}
	
	With \cite[Conjecture 2.5]{[Za19]} proven, cases \ref{situation item: K3} and \ref{situation item: Atiyah ruled} of Theorem \ref{theorem: main theorem} follow directly from \cite[Theorem 3.3]{[Za19]}, and in essence so does case \ref{situation item: Du Val}, since the formulation in terms of ${\mathbb P}^2$ given in \cite[Theorem 3.3]{[Za19]} is equivalent to the formulation in terms of blowups of ${\mathbb P}^2$ at $9$ points in Theorem \ref{theorem: main theorem}. Indeed,
	$$ {\sh L} = {\sh O}_S(3kH - kE_1 - \cdots -kE_8 - (k-1)E_9), $$
	where $H$ is the pullback of the class of a line in ${\mathbb P}^2$, and $E_1,\ldots,E_9 \subset S$ are the $9$ exceptional curves, since we may assume $E=E_9$ by Cremona transformations. 
			
	This leaves cases \ref{situation item: abelian surface} and \ref{situation item: split ruled}. If we construct suitable specializations of $(S,{\sh L})$ in situations \ref{situation item: abelian surface} and \ref{situation item: split ruled} satisfying the conditions in \cite[Proposition 3.1]{[Za19]}, then we are done. Such specializations exist, with the ultimately minor caveat that, in the abelian surface case, the base of the elliptic fibration as in loc. cit. is not ${\mathbb P}^1$, but an elliptic curve instead. Thus, we need a minor extension of \cite[Proposition 3.1]{[Za19]}. 
	
	\begin{proposition}\label{proposition: elliptic degeneration}
		Let $\pi:X \to B$ be a smooth projective family of surfaces over a smooth quasi-projective curve $B$, and let ${\sh L}_X \in {\mathrm{Pic}}(X)$ relatively ample over $B$. Let $X_t = \pi^{-1}(t)$ and ${\sh L}_t = {\sh L}_X|_{X_t}$ for any $t \in B$. Fix $b \in B$ closed. Assume that 
		\begin{equation}\label{cohomology vanishing}
		H^1(X_b,{\sh L}_b) = H^2(X_b,{\sh L}_b) = 0,
		\end{equation}
		and that the central fibre $X_b$ has an elliptic fibration $f:X_b \to Y$, where $Y$ is a smooth projective curve, with the following properties:
		\begin{enumerate}
		\item\label{condition 1} For general $s \in Y$, the natural short exact sequence 
$$ 0 \longrightarrow N_{f^{-1}(s)/X_b} \longrightarrow N_{f^{-1}(s)/X} \longrightarrow T_b B \otimes {\sh O}_{f^{-1}(s)} \longrightarrow 0 $$ 
on $f^{-1}(s)$ is not split;
		\item\label{condition 2} There exists a fixed section $G \subset X_b$ of $f$, such that any divisor $D \in |{\sh L}_b|$ is the sum of $G$ and $\dim |{\sh L}_b| + p_g(Y)$ (mobile) fibres of $f$.
		\end{enumerate}
		Let $t \in B$ general and $(S,{\sh L}) = (X_t,{\sh L}_t)$. If $m_1,m_2,...,m_n \geq 1$, then
		\begin{equation}\label{equation: expected dimension}
		\dim |{\sh L}(m_1,\ldots,m_n)| = \max \left\{-1, \dim |{\sh L}| - \sum_{i=1}^{n} \frac{m_i(m_i+1)}{2} \right\},
		\end{equation}
		with notation as in \eqref{equation: notation for linear system} for $|{\sh L}(m_1,\ldots,m_n)|$.
	\end{proposition}
	
	\begin{proof}
	The fact that the left hand side is at least the right hand side in \eqref{equation: expected dimension} (the dimension is at least the expected dimension) is well-known and trivial, so we only need to prove the reverse inequality. When $Y \simeq {\mathbb P}^1$, this is precisely the combination of \cite[Proposition 3.1]{[Za19]} and Proposition \ref{proposition: former conjecture}. In general, most of the argument in \cite{[Za19]} applies verbatim, though the conclusion requires an additional ingredient, which is precisely the $1$-dimensional interpolation (`SHGH') problem.
	
	\begin{lemma}\label{lemma: 1D interpolation}
		Let ${\sh M} \in \mathrm{Pic}(Y)$, and $p_1,\ldots,p_n \in Y$ \emph{general} points. Then
		$$ \dim |{\sh M}(-D)| = \max \{-1, \dim |{\sh M}| - \deg D \} $$
		for any effective divisor $D$ such that $\mathrm{Supp}(D) \subseteq \{p_1,\ldots,p_n\}$. 
	\end{lemma}
	
	\begin{proof}
	Generalizing to non-complete linear systems, the statement boils down to the case $n=1$, which is the well-known fact that, in characteristic $0$, a linear system on a curve has only finitely many inflection points, e.g. \cite[Proposition 1.1]{[EH83]}. 
	\end{proof}
	
	Let us explain how 1D interpolation must be applied at the end of the argument in \cite{[Za19]} to generalize the proof. Condition \ref{condition 2} implies that there exists ${\sh M} \in \mathrm{Pic}(Y)$ such that $|{\sh M}| \simeq |{\sh L}_b|$ by $D \mapsto G+f^{-1}(D)$ since $\pi^*:\mathrm{Pic}(Y) \to \mathrm{Pic}(X_b)$ is injective by \cite[III, Exercise 12.4]{[Ha77]} and the projection formula. We use the notation in the proof of \cite[Proposition 3.1]{[Za19]}. The next-to-last centred formula states that
	\begin{equation}
		P_b \subseteq \left\{ [D] \in |{\sh L}_b|: \mathrm{coeff}_{E_i} D \geq \lambda(m_i) = \frac{m_i(m_i+1)}{2} \right\}=:T,
	\end{equation}
	that is, the limit in $X_b$ of divisors with high multiplicities at the chosen points must contain the elliptic fibre $E_i$ with multiplicity at least $\lambda(m_i)$, for $i=1,\ldots,n$; please see also Remark \ref{remark: blowing up fibres} below. Indeed, the argument in loc. cit. up to that point requires no changes, and the formula for $\lambda(m_i)$ is our Proposition \ref{proposition: former conjecture}. Under $|{\sh M}| \simeq |{\sh L}_b|$, $T$ corresponds to $ T_Y =  \left\{ D \in |{\sh M}|: \mathrm{coeff}_{p_i} D \geq \lambda(m_i) \right\}$. Then
	\begin{equation*} 
	\begin{aligned} 
	\dim P_b \leq \dim T = \dim T_Y &=  \max \left\{ -1, \dim |{\sh M}| - \sum_{i=1}^{n} \frac{m_i(m_i+1)}{2} \right\} \\
	&= \max \left\{ -1, \dim |{\sh L}_t| - \sum_{i=1}^{n} \frac{m_i(m_i+1)}{2} \right\}
	\end{aligned}
	\end{equation*}
	by Lemma \ref{lemma: 1D interpolation} and $\dim |{\sh M}| = \dim |{\sh L}_b| = \dim|{\sh L}_t|$, and we may conclude as in \cite{[Za19]}. (The situation is thus slightly different when $Y \not\simeq {\mathbb P}^1$, in that the limit divisors are forced to contain some elliptic fibres other than $E_1,\ldots,E_n$, which plays a role in obtaining the correct bound).
	\end{proof}
	
	\begin{remark}\label{remark: blowing up fibres}
	For the reader's convenience, we review that the argument in \cite{[Za19]} entails blowing up the elliptic fibres which contain the (reduced) limits of the fat points in the central fibre. This is reminiscent of the approach of Ciliberto and Miranda to the SHGH Conjecture using degenerations \cite{[CM98], [CM00]}. In our situation, the exceptional divisors of such blowups are isomorphic to the ruled surface in \S\ref{S: 2.2}.
	\end{remark}
	
	It remains to construct specializations with the desired properties.
	
	\subsubsection*{Ruled surfaces (\ref{situation item: split ruled})} It is clear that the group of automorphisms of the surface acts transitively on the set of linear systems we are considering, so all such linear systems behave identically. Let $B = \mathrm{Pic}^0(E)$. Let ${\sh P}$ be the universal (Poincar\'{e}) line bundle over $B \times E = \mathrm{Pic}^0(E) \times E$. Then we define $X = {\mathbb P} ({\sh O}_{B \times E} \oplus {\sh P} )$ and $o = [{\sh O}_E] \in B$. There are two distinguished sections $\Sigma_0,\Sigma_\infty \subset X$ corresponding to the projections to the two terms of ${\sh O}_{B \times E} \oplus {\sh P}$. Let $q \in E$ and $k \geq 1$ integer, and 
	$$ {\sh L}_X = {\sh O}_X(k\Sigma_\infty + \pi^{-1}(B \times \{q\})). $$ 
	Then $X_o = E \times {\mathbb P}^1$, and condition \ref{condition 2} follows easily from the fact that any effective divisor on $E \times {\mathbb P}^1$ with intersection number $1$ with $E \times \{\mathrm{point}\}$ is the sum of fibres of projections to the two factors, which is elementary.
	
	Condition \eqref{cohomology vanishing} follows from the Kodaira vanishing theorem.
		
	To check the condition \ref{condition 1}, let $R \simeq \mathrm{Spec}\mathop{} {\mathbb C}[\epsilon]/(\epsilon^2)$ be the first order thickening of $o=[{\sh O}_E]$ in $B$. Let $y \in {\mathbb P}^1$ such that $E \times \{y\} \not\subset \Sigma_0,\Sigma_\infty$, and let's assume by way of contradiction that the sequence 
	\begin{equation}\label{equation: common ses} 0 \longrightarrow N_{E \times \{y\},X_o} \longrightarrow N_{E \times \{y\},X} \longrightarrow N_{X_o,X}|_{E \times \{y\}} \longrightarrow 0 \end{equation}
	was split. Then $H^0(N_{E \times \{y\},X_o}) \hookrightarrow H^0(N_{E \times \{y\},X})$ is not surjective, and the global sections of $H^0(N_{E \times \{y\},X})$ not coming from global sections of $H^0(N_{E \times \{y\},X_o})$ give a first order deformation of $E \times \{y\} \subset X$ flat over $R$. This deformation is a section of $X|_{R \times E}$, so corresponds to a surjective map $\nu:{\sh O}_{R \times E} \oplus {\sh P}|_{R \times E} \to {\sh Q}$ to a line bundle ${\sh Q}$ on $R \times E$. Clearly, $\deg {\sh Q}_o = 0$. Since we are assuming $E \times \{y\} \not\subset \Sigma_0,\Sigma_\infty$, the restrictions of $\nu$ to ${\sh O}_{R \times E} \oplus 0$ and $0 \oplus {\sh P}|_{R \times E}$ are nonzero on $\{o\} \times E$, and are therefore isomorphisms for degree reasons. Hence, ${\sh Q} \simeq {\sh O}_{R \times E}$ and ${\sh Q} \simeq {\sh P}|_{R \times E}$. However, ${\sh P}|_{R \times E} \not\simeq {\sh O}_{R \times E}$ by definition, which is a contradiction.
	
	\subsubsection*{Abelian surfaces} Let $X \to B$ be a one-parameter family of $(1,d)$-polarized abelian surfaces specializing to a product $E \times F$ of two elliptic curves, as in \cite{[Za22]} (or \cite[\S4]{[BL99]} implicitly). Let us be more specific. If $E \times F$ is polarized by ${\sh J} = {\sh J}_E \boxtimes {\sh J}_F$, with $\deg {\sh J}_E = 1$ and $\deg {\sh J}_F =d$, then $(E \times F, {\sh J})$ is a $(1,d)$-polarized abelian surface. Let ${\modu A}_d$ be the Deligne-Mumford moduli stack of $(1,d)$-polarized abelian surfaces, and $(U,u) \to {\modu A}_d$ an affine (in particular, quasi-projective) \'{e}tale neighbourhood of $(E \times F, {\sh J})$. By the usual Kodaira-Spencer theory, there is a natural injective map
	$$ T_uU = T_{(E \times F, {\sh J})} {\modu A}_d \hookrightarrow H^1(E \times F,{\sh T}_{E \times F}) \simeq {\mathbb C}^4. $$
	By \cite[Lemma 2.2]{[Ch02]}, the image is contained in the annihilator of $c_1({\sh J}) \in H^{1,1}(E \times F)$ relative to the natural pairing
	$$ H^1(E \times F,{\sh T}_{E \times F}) \times H^1(E \times F,\Omega_{E \times F}) \longrightarrow {\mathbb C}, $$
	and hence must coincide with it since $\dim T_uU \geq \dim U = 3$. (To apply \cite[Lemma 2.2]{[Ch02]}, we need to consider a curve through $u$ in $U$.) Then we take $(B,o) \to (U,u)$ to be a smooth curve through $u$ corresponding to a \emph{general} direction $T_oB \hookrightarrow T_uU$, this is possible since $U$ is quasi-projective. (In \cite{[Za22]}, this \emph{first order} generality was not really used and thus not imposed, but here it is needed.)
	
	Condition \eqref{cohomology vanishing} again follows from the Kodaira vanishing theorem. Condition \ref{condition 2} follows from the fact that any divisor $D \in |{\sh J}|$ is of the form 
	$$ D = E \times D' + z \times F, \quad \text{where $|D'| \in |{\sh J}_F|$ and $\{z\} = |{\sh J}_E|$}, $$
which is elementary (e.g. \cite{[BL99]} or \cite{[Za22]}). It remains to check condition \ref{condition 1}, namely, that the short exact sequence notationally identical to \eqref{equation: common ses} is not split. Here, $y \in F$ is completely arbitrary. The proof is very similar to the proof of \cite[Proposition 2.1]{[Ch02]}, so we will be brief. It suffices to check that $H^0(N_{X_o,X}|_{E \times \{y\}}) \to H^1(N_{E \times \{y\},X_o})$ is nonzero. As in \cite[(2.10)]{[Ch02]}, this map factors as
		\begin{equation}\label{equation: composition}
		\begin{aligned}
		H^0(N_{X_o,X}|_{E \times \{y\}}) \simeq T_oB \xrightarrow{\mathrm{ks}} H^1({\sh T}_{E \times F}) &\longrightarrow H^1({\sh T}_{E \times F}|_{E \times \{y\}}) \\ &\longrightarrow H^1(N_{E \times \{y\},X_o}),
		\end{aligned}
		\end{equation}
		where $\mathrm{ks}$ is the Kodaira-Spencer map. It is clear that the composition of the last two maps in \eqref{equation: composition} is surjective, and its kernel is different from the annihilator of $c_1({\sh J})$ considered above -- for instance, under the natural identification $H^1({\sh T}_{E \times F}) \simeq {\mathbb C}^4$, the former is clearly a `coordinate hyperplane', which the latter is not.
		Then the generality of the Kodaira-Spencer map (i.e. of the tangent direction in $U$) implies that \eqref{equation: composition} is nonzero, as desired.
		
		This completes the proof of Theorem \ref{theorem: main theorem}.
		
		\begin{remark}\label{remark: De Volder Laface}
		Theorem \ref{theorem: main theorem} rather vacuously implies the primitive case ($d=1$) of \cite[Conjecture 2.1]{[DL05a]}. Then the primitive case of \cite[Conjecture 2.3]{[DL05a]} follows from \cite[Theorem 3.7]{[DL05a]}. To be pedantic, the equivalence (since one implicitation is trivial) of the two conjectures proved in loc. cit. does not \emph{logically} imply their equivalence for $d=1$, though the proof clarifies that there is absolutely no issue. 
		\end{remark}
	
	\bibliographystyle{plain}
\bibliography{SHGH-KP-bibtex}{}
	
	\end{document}